\documentclass[a4paper,12pt]{amsart}
\usepackage[T1]{fontenc}
\usepackage{amsmath,amsthm}

\makeatletter
\newcommand{\subjclassname@later}{\textup{2010} Mathematics Subject Classification}
\makeatother

\newtheorem{theorem}{Theorem}[section]
\newtheorem{lemma}{Lemma}[section]
\newtheorem{corollary}{Corollary}[section]

\theoremstyle{definition}
\newtheorem{definition}{Definition}
\newtheorem{remark}{Remark}[section]

\numberwithin{equation}{section} \numberwithin{theorem}{section}

\DeclareMathOperator{\RE}{Re}

\title[Certain Close-to-Convex Functions]{Coefficient, Distortion  and Growth Inequalities for Certain Close-to-Convex Functions}

\author[N. E. Cho]{Nak Eun Cho}
\address{Department of Applied Mathematics\\ Pukyong National University, Busan 608-737, South Korea}
\email{necho@pknu.ac.kr}

\author[O. S. Kwon]{Oh Sang Kwon}
\address{Department of Mathematics\\
Kyungsung University, Busan 608-736, South
Korea}\email{oskwon@ks.ac.kr}

\author[V. Ravichandran]{V. Ravichandran}
\address{Department of Mathematics\\ University of Delhi,
Delhi-110007, India} \email{vravi@maths.du.ac.in}

\thanks{This work was completed during the  third author's visit to
Pukyong National University and the support and the hospitality of
Prof. Cho is gratefully acknowledged.}
\begin{document}
\begin{abstract}
In the present investigation, certain subclasses of close-to-convex
functions are investigated. In particular, we obtain an estimate for
the Fekete-Szeg\"{o} functional for functions belonging to the
class, distortion,  growth estimates and covering theorems.
\end{abstract}
\maketitle

\section{Introduction}Let $\mathbb{D}:=\{z\in\mathbb{C}:|z|<1\}$ be
the open unit disk in the complex plane $\mathbb{C}$. Let
$\mathcal{A}$ be the class of analytic functions defined on
$\mathbb{D}$ and normalized by the conditions $f(0)=0$ and
$f'(0)=1$. Let $\mathcal{S}$ be the subclass of $\mathcal{A}$ consisting of
univalent functions. Sakaguchi \cite{saka}  introduced a class of functions
called  starlike functions with respect to symmetric points; they
are the functions $f\in\mathcal{A}$  satisfying the condition
\[\RE \frac{zf'(z)}{f(z)-f(-z)} > 0 .\]  These functions are
close-to-convex functions. This can be easily seen by showing that
the function  $(f(z)-f(-z))/2$  is a starlike function in
$\mathbb{D}$. Motivated by the class of  starlike functions with
respect to symmetric points,  Gao and Zhou \cite{gao2} discussed a
class $\mathcal{K}_s$ of close-to-convex univalent functions.  A
function $f\in \mathcal{K}_s$ if it satisfy  the following
inequality \[ \RE\left( \frac{ z^2f'(z)}{ g(z)g(-z)}\right) < 0
\quad (z \in \mathbb{D})\] for some function
 $ g\in  S^*( 1 /2 )$.
 The idea here is to replace the average of $f(z)$ and $-f(-z)$ by the
 corresponding product $-g(z)g(-z)$ and the factor $z$ is included
 to normalize the expression  so that $-z^2f'(z)/(g(z)g(-z))$ takes
 the value 1 at $z=0$. To make the functions  univalent, it is
 further assumed that $g$ is starlike  of order 1/2 so that  the function
 $-g(z)g(-z)/z$ is starlike which in turn implies the
 close-to-convexity of $f$. For some recent works on the  problem, see
 \cite{gao1,gao3,gao4,xu}. In stead of requiring the quantity
 $-z^2f'(z)/(g(z)g(-z))$ to lie in the right-half plane, we can
 consider more general regions. This  could be done via
 subordination between analytic functions.

Let $f$ and $g$ be analytic in $\mathbb{D}$. Then $f$ is subordinate
to $g$, written $f\prec g$ or $f(z)\prec g(z)$ $(z\in \mathbb{D})$,
if there is an analytic function $w(z)$, with $w(0)=0$ and
$|w(z)|<1$, such that $f(z)=g(w(z))$. In particular, if $g$ is
univalent in $\mathbb{D}$, then $f$ is subordinate to $g$, if
$f(0)=g(0)$ and $f(\mathbb{D})\subseteq g(\mathbb{D})$. In terms of
subordination, a general class $\mathcal{K}_s(\varphi)$ is
introduced in  the following  definition.

\begin{definition}\cite{gao1}
For a function $\varphi$ with positive real part, the class
$\mathcal{K}_s(\varphi)$ consists of functions $f\in \mathcal{A}$
satisfying
\begin{equation}\label{def1}  - \frac{ z^2f'(z)}{ g(z)g(-z)}\prec \varphi(z) \quad (z \in
\mathbb{D})\end{equation} for some function  $ g\in  S^*( 1 /2 )$.
\end{definition}

This class  was introduced by  Wang, Gao  and Yuan \cite{gao1}. A
special subclass   $\mathcal{K}_s(\gamma):=\mathcal{K}_s(\varphi)$
where $\varphi(z):=(1+(1-2\gamma)z)(1-z) $, $0\leq \gamma<1$, was
recently investigated by Kowalczyk and    Le\'s-Bomba \cite{kow}.
They proved the  sharp distortion and growth estimates for functions
in $\mathcal{K}_s(\gamma)$ as well as some sufficient conditions in
terms of the coefficient for function to be in this class
$\mathcal{K}_s(\gamma)$.

In the present investigation,  we obtain a sharp estimate for the
Fekete-Szeg\"{o} functional for functions belonging to the class
$\mathcal{K}_s(\varphi)$. In addition, we also investigate the
corresponding problem for  the inverse functions for functions
belonging to the  class  $\mathcal{K}_s(\varphi)$.  Also distortion,
growth estimates as well as covering theorem are derived. Some
connection with earlier works are also indicated.

\section{Fekete-Szeg\"{o} Inequality} In this section, we assume that the
function $\varphi(z)$ is an univalent analytic function with positive
real part that maps the unit disk $\mathbb{D}$ onto a starlike
region which symmetric with respect to real axis and is normalized
by $\varphi(0)=1$ and $\varphi'(0)>0$. In such case, the function
$\varphi$ has an expansion of the form
$\varphi(z)=1+B_1z+B_2z^2+\cdots$, $B_1>0$.

\begin{theorem}[Fekete-Szeg\"{o} Inequality]\label{th1}
For a function $f(z)=z+a_2z^2+a_3z^3+\cdots$ belonging to the class
$\mathcal{K}_s(\varphi)$, the following sharp estimate holds:
\[  |a_3-\mu a_2^2|\leq 1/3+\max(B_1/3, |B_2/3-\mu B_1^2/4|)\quad (\mu\in\mathbb{C}). \]
\end{theorem}

\begin{proof}Since the function $f\in \mathcal{K}_s(\varphi)$, there is a
normalized  analytic  function $g\in S^*(1/2)$ such that
\[- \frac{ z^2f'(z)}{g(z)g(-z)}\prec \varphi(z).\]  By using the
definition of subordination between analytic function, we find a
function $w(z)$ analytic in $\mathbb{D}$, normalized by $w(0)=0$
satisfying $|w(z)|<1$ and
\begin{equation}\label{p1e1}
- \frac{ z^2f'(z)}{g(z)g(-z)}= \varphi(w(z)).
\end{equation}
By writing $w(z)=w_1 z+w_2 z^2+\cdots$, we see that
\begin{equation}\label{p1e2}
\varphi(w(z))= 1+B_1w_1  z+(B_1w_2+B_2 w_1^2)z^2+\cdots.
\end{equation}
Also by writing $g(z)=z+g_2z^2+g_3z^3+\cdots$,  a calculation shows
that
\[ -\frac{g(z)g(-z)}{z} = 1+(2g_3-g_2^2)z^3+\cdots\] and therefore
\[ -\frac{z}{g(z)g(-z)}  = 1-(2g_3-g_2^2)z^3+\cdots.\]
Using this and the Taylor's  expansion for $zf'(z)$, we get
\begin{equation}\label{p1e3}
- \frac{ z^2f'(z)}{g(z)g(-z)} = 1+2a_2z
+(3a_3-2g_3+g_2^2)z^2+\cdots.
\end{equation}
Using  \eqref{p1e1},  \eqref{p1e2} and \eqref{p1e3}, we see that
\begin{align*}
2a_2& = B_1w_1,\\
3a_3 &= 2g_3-g_2^2+B_1w_2+B_2w_1^2.
\end{align*}
This shows that
\[a_3-\mu a_2^2 = (2/3)(g_3-g_2^2/2)+(B_1/3)\left(w_2+(B_2/B_1-3\mu
B_1/4)w_1^2\right).\] By using the following  estimate
(\cite[inequality 7, p.\ 10]{keo}) \[ |w_2- t w_1^2| \leq  \max\{1;\
| t | \}  \quad (t\in\mathbb{C}) \] for an analytic  function $w$
with $w(0)=0$ and $|w(z)|<1$
 which is
sharp for the functions $w(z)=z^2$ or $w(z)=z$, the desired result
follows upon using the estimate that $|g_3-g_2^2/2|\leq 1/2$ for
analytic function $g(z)=z+g_2z^2+g_3z^3+\cdots$ which is starlike of
order 1/2.

Define the function $f_0$ by
\[f_0(z) = \int_0^z \frac{\varphi(w)}{1-w^2}dw. \] The function
clearly belongs to the class $\mathcal{K}_s(\varphi)$ with
$g(z)=z/(1-z)$. Since \[\frac{ \varphi(w)}{1-w^2} = 1+B_1
w+(B_2+1)w^2+\cdots,\] we have
\[ f_0(z)= z+ (B_1/2) z^2 +(1/3+B_2/3)z^3+\cdots.\]
Similarly, define $f_1$ by
\[f_1(z) = \int_0^z \frac{\varphi(w^2)}{1-w^2}dw. \]
Then
\[ f_1(z)= z+ (B_1/3+1/3)z^3+\cdots.\]
The functions $f_0$ and $f_1$ show that the results are  sharp.
\end{proof}

\begin{remark}
By setting $\mu=0$ in Theorem~\ref{th1}, we  get the sharp estimate
for the third coefficient of functions in $\mathcal{K}_s(\varphi)$:
\[  |a_3|\leq 1/3+(B_1/3)\max(1, |B_2|/B_1), \]
while the limiting case $\mu\rightarrow \infty$ gives the sharp
estimate $|a_2|\leq B_1/2$. In the special  case where
$\varphi(z)=(1+z)/(1-z)$, the results reduce to the corresponding
one in \cite[Theorem 2, p.\ 125]{gao2}.

Though Xu et al.\ \cite{xu}  have given an estimate of $|a_n|$ for all $n$, their result is not sharp in general. For 
$n=2,3$, our results provide sharp bounds. 
\end{remark}

It is known that every univalent function $f$ has an inverse
$f^{-1}$, defined by
\[f^{-1}(f(z))=z, \quad z\in \mathbb{D}\] and
\[f(f^{-1}(w))=w, \quad \left(|w|<r_0(f);
r_0(f)\geq\frac{1}{4}\right).\]

\begin{corollary}Let $f\in  \mathcal{K}_s(\varphi)$. Then the coefficients $d_2$ and $d_3$ of the
inverse  function $f^{-1}(w)=w+d_2w^2+d_3w^3+\cdots$ satisfy the
inequality
\[  |d_3-\mu d_2^2|\leq 1/3+\max(B_1/3, |B_2/3-(2-\mu) B_1^2/4|)\quad (\mu\in\mathbb{C}). \]
\end{corollary}

\begin{proof} A calculation shows that the inverse function $f^{-1}$
has the following Taylor's series expansion:
\[f^{-1}(w)=w+a_2w^2+(2a_2^2-a_3)w^3+\cdots.\]
From this expansion, it follows that $d_2=a_2$ and $d_3=2a_2^2-a_3$
and hence
\[ |d_3-\mu d_2^2|= |a_3-(2-\mu)a_2^2|. \] Our result follows at
once from this identity and Theorem~\ref{th1}.
\end{proof}

\section{Distortion and growth theorems}
The second coefficient of univalent function plays an important role
in the theory of univalent function; for example, this leads to  the
distortion and growth estimates for univalent functions as well as
the rotation theorem. In the next theorem, we derive the distortion
and growth estimates for the functions in the class
$\mathcal{K}_s(\varphi)$. In particular, if we let $r\rightarrow 1-$
in the growth estimate, it gives the bound $|a_2|\leq B_1/2$ for the
second coefficient of functions in $\mathcal{K}_s(\varphi)$.

\begin{theorem}Let $\varphi$ be an analytic univalent functions with
positive real part and  \[ \phi(-r)=\min_{|z|=r<1} |\phi(z)|, \quad
 \phi(r)=\max_{|z|=r<1} |\phi(z)|.\] If $f\in \mathcal{K}_s(\varphi)$, the
 following sharp inequalities holds:
 \[\frac{\varphi(-r)}{1+r^2} \leq |f'(z)| \leq  \frac{\varphi(r)}{1-r^2}\quad (|z|=r<1),\]
  \[\int_0^r\frac{\varphi(-t)}{1+t^2} dt\leq |f(z)|
  \leq \int_0^r \frac{\varphi(t)}{1-t^2}dt\quad (|z|=r<1).\]
\end{theorem}

\begin{proof}Since the function $f\in \mathcal{K}_s(\varphi)$, there is a
normalized  analytic  function $g\in S^*(1/2)$ such that
\begin{equation}\label{p2e1} - \frac{ z^2f'(z)}{g(z)g(-z)}\prec \varphi(z).\end{equation} Define
the function $G:\mathbb{D}\rightarrow \mathbb{C}$ by the equation
\[ G(z):= -\frac{g(z)g(-z)}{z}. \] Then it is clear that $G$ is odd
starlike function in $\mathbb{D}$ and therefore
\[ \frac{r}{1+r^2} \leq |G(z)| \leq  \frac{r}{1-r^2}\quad (|z|=r<1) \]
Using the definition of subordination between analytic function, and
the equation \eqref{p2e1}, we see that there is an analytic function
$w(z)$ with $|w(z)|\leq |z|$ such that
\[\frac{zf'(z)}{G(z)}=\varphi(w(z))\]
or $zf'(z)=G(z)\varphi(w(z))$. Since $w(\mathbb{D})\subset
\mathbb{D}$, we have, by maximum principle for harmonic functions,
\[ |f'(z)| \leq \frac{|G(z)|}{|z|} |\varphi(w(z))| \leq
\frac{1}{1-r^2}\max_{|z|=r}|\varphi(z)|  = \frac{\varphi(r)}{1-r^2}.
\]  The other inequality for $|f'(z)|$ is similar. Since the function $f$ is
univalent, the inequality for $|f(z)|$ follows from the
corresponding inequalities for $|f'(z)|$ by Privalov's Theorem
\cite[Theorem 7, p.\ 67]{good1}.

To prove the sharpness of our  results, we consider  the functions
\begin{equation}\label{sharp}  f_0(z)= \int_0^z \frac{\varphi(w)}{1- w^2} dw,\quad
 f_1(z)= \int_0^z \frac{\varphi(w)}{1+ w^2} dw . \end{equation}
Define the function $g_0$ and $g_1$ by  $g_0(z)=z/(1-z)$ and
$g_1(z)=z/\sqrt{1+z^2}$. These functions are clearly starlike
functions of order 1/2. Also a calculations shows that
\[ -\frac{z^2f_k'(z)}{g_k(z)g_k(-z)} = \varphi(z) \quad (k=0,1).  \]
Thus  the function $f_0$ satisfies the subordination \eqref{def1}
with $g_0$  while the function $f_1$ satisfies it with $g_1$;
therefore, these functions belong to the class
$\mathcal{K}_s(\varphi)$. It is clear that the upper estimates for
$|f'(z)|$ and $|f(z)|$ are sharp for the function $f_0$ given in
\eqref{sharp} while the lower estimates  are sharp for $f_1$ given
in \eqref{sharp}.
\end{proof}

\begin{remark}
We  note that Xu et al.\ \cite{xu} also obtained a similar estimates and 
our results differ from their in the  hypothesis. Also we have shown that the 
results are sharp.  Our hypothesis is same as the one 
assumed by Ma and Minda \cite{mamin}. 
\end{remark} 

\begin{remark}For the choice $\varphi(z)=(1+z)/(1-z)$, our result
reduces to \cite[Theorem 3, p. 126]{gao2} while, for the choice
 $\varphi(z)=(1+(1-2\gamma)z)/(1-z)$, it reduces to following  estimates (obtained in \cite[Theorem
 4,  p. \ 1151]{kow}) for $f\in \mathcal{K}_s(\gamma)$:
\[ \frac{1-(1-2\gamma)r}{(1+r)(1+r^2)} \leq |f'(z)| \leq \frac{1+(1-2\gamma)r}{(1-r)(1-r^2)} \]
and
\[ (1-\gamma) \ln \frac{1+r}{\sqrt{1+r^2}}+\gamma  \arctan r \leq |f(z)| \leq
\frac{\gamma}{2}\ln\frac{1+r}{1-r}+(1-\gamma)\frac{r}{1-r}\] where
$|z|=r<1$.  Also our result improves the corresponding
 results in \cite{gao1}.
\end{remark}

\begin{remark}Let $k: = \lim_{r\rightarrow 1-}\int_0^r\varphi(-t)/(1+t^2) dt$. Then
the disk $\{ w\in\mathbb{C}: |w|\leq k\} \subseteq f(\mathbb{D})$
for every $f\in \mathcal{K}_s(\varphi)$.
\end{remark}

\section{A Subordination Theorem}

It is well-known \cite{stan}  that $f$ is starlike if  $(1-t)f(z)\prec f(z)$ for $t\in(0, \delta)$, where $\delta$ is a positive real number; also the function
is starlike with respect to symmetric points if $(1-t)f(z)+tf(-z)\prec f(z)$. In the following theorem, we extend these results to the class $\mathcal{K}_s$.
The proof of our result is based on  the following  version of a lemma of Stankiewicz \cite{stan}.

\begin{lemma}\label{lemstan}
Let $F(z,t)$ be analytic in $\mathbb{D}$ for each $t\in (0,\delta)$, $F(z,0)=f(z)$, $f\in\mathcal{S}$ and
$F(0,t)=0 $ for each  $t\in (0,\delta)$. Suppose that $F(z,t)\prec f(z)$ and that
\[ \lim_{t\rightarrow 0^+} \frac{F(z,t)-f(z)}{zt^\rho}=F(z) \]
exists for some $\rho >0$. If $F$ is analytic and $\RE(F(z))\neq 0$, then
\[ \RE\left( \frac{F(z)}{f'(z)} \right)<0. \]
\end{lemma}

\begin{theorem}Let $f\in  \mathcal{S}$ and   $g\in \mathcal{S}^*(1/2)$. Let $\delta>0$ and $f(z)+tg(z)g(-z)/z \prec f(z)$, $t\in (0, \delta)$.  Then $f\in \mathcal{K}_s$.
\end{theorem}

\begin{proof} Define the function $F$ by $F(z,t)=f(z)+tg(z)g(-z)/z$. Then $F(z,t)$ is analytic for every fixed $t$ and $F(z,0)=f(z)$ and by our assumption, $f\in \mathcal{S}$. Also
\[ \lim_{t\rightarrow 0^+} \frac{F(z,t)-f(z)}{z}=\frac{g(z)g(-z)}{z}:=F(z). \]
The function $F$ is analytic in $\mathbb{D}$ (of course, one has to redefine the function $F$ at $z=0$ where it has removable singularity.)
Since all the hypothesis of Lemma~\ref{lemstan} are satisfied, we have
\[ \RE\left( \frac{g(z)g(-z)}{z^2f'(z)}\right) <0 .\]
Since a function $p(z)$ has negative real part if and only if its reciprocal $1/p(z)$ has negative real part,
we have
\[ \RE\left( \frac{z^2f'(z)}{g(z)g(-z)}\right) < 0 .\]
Thus, $f\in \mathcal{K}_s$.
\end{proof}

\end{document}